\documentclass[preprint,12pt]{elsarticle}

\usepackage{amsmath, amsthm, amsfonts, amssymb}
\usepackage{hyperref}

\journal{the arXiv}

\newtheorem{theorem}{Theorem}[section]
\newtheorem{lemma}[theorem]{Lemma}

\theoremstyle{definition}

\theoremstyle{remark}

\numberwithin{equation}{section}

\numberwithin{equation}{section}

\begin{document}

\begin{frontmatter}

\title{A new version of the second main theorem for meromorphic mappings intersecting hyperplanes in several complex variables\tnoteref{mytitlenote}}
\tnotetext[mytitlenote]{The first author was supported in part by NSFC(no.11461042),  CPSF(no.2014M551865), CSC(no.201308360070),  PSF of Jiangxi(no.2013KY10). The second author was supported in part by the Academy of Finland grant(\#286877) and (\#268009).}



\author{Tingbin Cao}
\address{Department of Mathematics, Nanchang University, Nanchang, Jiangxi 330031, P. R. China}
\ead{tbcao@ncu.edu.cn}
\author{Risto Korhonen}
\address{Department of Physics and Mathematics, University of Eastern Finland, P.O. Box 111,
FI-80101 Joensuu, Finland}
\ead{risto.korhonen@uef.fi}

\begin{abstract}
Let $c\in \mathbb{C}^{m},$ $f:\mathbb{C}^{m}\rightarrow\mathbb{P}^{n}(\mathbb{C})$ be a linearly nondegenerate meromorphic mapping over the field $\mathcal{P}_{c}$ of $c$-periodic meromorphic functions in $\mathbb{C}^{m}$, and let $H_{j}$ $(1\leq j\leq q)$ be $q(>2N-n+1)$ hyperplanes in $N$-subgeneral position of $\mathbb{P}^{n}(\mathbb{C}).$ We prove a new version of the second main theorem for meromorphic mappings of hyperorder strictly less than one without truncated multiplicity by considering the Casorati determinant of $f$ instead of its Wronskian determinant. As its applications,  we obtain a defect relation, a uniqueness theorem and a difference analogue of generalized Picard theorem.
\end{abstract}


\end{frontmatter}

\allowdisplaybreaks

\section{Introduction}

The Picard's theorem says that all holomorphic mappings $f: \mathbb{C}^{1}\rightarrow \mathbb{P}^{1}(\mathbb{C})\setminus\{a, b, c\}$ are constants. Since Nevanlinna \cite{nevanlinna:25} established the second main theorem for meromorphic functions in the complex plane in 1925 and Ahlfors did it for meromorphic curves in 1941, many forms of the second main theorem for holomorphic maps, as well as meromorphic maps, on various contexts were found. They are powerful generalizations of the Picard's theorem, and are also applied to defect relations and uniqueness problems. By Weyl-Ahlfors' method Chen \cite{chen:90} proved a second main theorem as follows. The case of $m=1$ is proved by H. Cartan \cite{cartan:33} when hyperplanes $H_{j}$ $(1\leq j \leq q)$ are in general position. \par

\begin{theorem}[\cite{chen:90,cartan:33}]\label{T-1.0} Let $f:\mathbb{C}^{m}\rightarrow\mathbb{P}^{n}(\mathbb{C})$ be a linearly nondegenerate meromorphic mapping over $\mathbb{C}^{1},$ and let $H_{j}$ $(1\leq j\leq q)$ be $q(>2N-n+1)$ hyperplanes in $N$-subgeneral position in $\mathbb{P}^{n}(\mathbb{C}).$ Then we have
\begin{eqnarray*}
 (q-2N+n-1)T_{f}(r)\leq\sum_{j=1}^{q}N(r, \nu_{(f, H_{j})}^{0})-\frac{N+1}{n+1}N(r, \nu_{W(f)}^{0})+o(T_{f}(r))
\end{eqnarray*} for all $r>0$ outside of a possible exceptional set $E\subset[1, +\infty)$ of finite Lebesgue measure, where $W(f)$ is the Wronskian determinant of $f.$ \end{theorem}

Let $c\in \mathbb{C}^{m}.$ Throughout this paper, we denote by $\mathcal{M}_{m}$ the set of all meromorphic functions on $\mathbb{C}^{m},$ by $\mathcal{P}_{c}$ the set of all meromorphic functions of $\mathcal{M}_{m}$ periodic with period $c,$ and by $\mathcal{P}_{c}^{\lambda}$ the set of all meromorphic functions of $\mathcal{M}_{m}$ periodic with period $c$ and having their hyperorders strictly less than $\lambda.$ Obviously, $\mathcal{M}_{m}\supset\mathcal{P}_{c}\supset\mathcal{P}_{c}^{\lambda}.$\par

In 2006, R. G. Halburd and R. J. Korhonen \cite{halburdk:06AASFM} considered the second main theorem for complex difference operator with finite order in the complex plane. Later, in \cite{wonglw:09} and \cite[Theorem 2.1]{halburdkt:14} difference analogues of the second main theorem for holomorphic curves  in $\mathbb{P}^{n}(\mathbb{C})$ were obtained independently, and in \cite[Theorem 3.3]{korhonen:12} and \cite[Theorems 1.6, 1.7]{cao:14} difference analogues of the second main theorem for meromorphic functions on $\mathbb{C}^{m}$ were obtained. In this paper, we will obtain a new natural difference analogue of Theorem \ref{T-1.0}, in which the counting function $N(r, \nu_{W(f)}^{0})$ of the Wronskian determinant of $f$ is replaced by the counting function $N(r, \nu_{C(f)}^{0})$ of the Casorati determinant of $f$ (it was called the finite difference Wronskian determinant in \cite{wonglw:09}). The hyperorder $\zeta_{2}(f)$ of meromorphic mapping $f:\mathbb{C}^{m}\rightarrow\mathbb{P}^{n}(\mathbb{C})$ is strictly less than one.\par

\begin{theorem}\label{T-3.1}  Let $c\in\mathbb{C}^{m},$ let $f:\mathbb{C}^{m}\rightarrow\mathbb{P}^{n}(\mathbb{C})$ be a linearly nondegenerate meromorphic mapping over $\mathcal{P}_{c}$ with hyperorder $\zeta=\zeta_{2}(f)<1,$ and let $H_{j}$ $(1\leq j\leq q)$ be $q(>2N-n+1)$ hyperplanes in $N$-subgeneral position in $\mathbb{P}^{n}(\mathbb{C}).$ Then we have
\begin{eqnarray*}
(q-2N+n-1)T_{f}(r)\leq\sum_{j=1}^{q}N(r, \nu_{(f, H_{j})}^{0})-\frac{N}{n}N(r, \nu_{C(f)}^{0})+o\left(\frac{T_{f}(r)}{r^{1-\zeta-\varepsilon}}\right)
\end{eqnarray*}
for all $r>0$ outside of a possible exceptional set $E\subset[1, +\infty)$ of finite logarithmic measure, where $C(f)$ is the Casorati determinant of $f.$
\end{theorem}

The remainder of this paper is organized in the following way. In Section~\ref{prelimsec}, some notations and basic results of Nevanlinna theory are introduced briefly. In Section~\ref{proofsec}, we adopt the Cartan-Nochka's method \cite{noguchi:05} and use the Casorati determinant to prove Theorem \ref{T-3.1}, from which a defect relation is obtained in Section~\ref{defectsec}. In Section~\ref{uniquenesssec}, we show a uniqueness theorem for meromorphic mappings intersecting hyperplanes in $N$-subgeneral position with counting multiplicities, which can be seen as a Picard-type theorem, and will be proved as a special case from a difference analogue of generalized Picard theorem \cite{fujimoto:72b,green:72} in Section~\ref{picardsec}.

\section{Preliminaries}\label{prelimsec}

\noindent\textbf{2.1.} Set $\|z\|=(|z_{1}|^{2}+\cdots+|z_{m}|^{2})$ for $z=(z_{1},\cdots,z_{m})\in\mathbb{C}^{m},$  for $r>0,$ define
\begin{equation*}
B_{m}(r):={\{z\in\mathbb{C}^{m}:\|z\|\leq r}\},\quad
S_{m}(r):={\{z\in\mathbb{C}^{m}:\|z\|=r}\}. \end{equation*}
Let $d=\partial+\overline{\partial}, \quad d^{c}=(4\pi\sqrt{-1})^{-1}(\partial+\overline{\partial}).$ Write
\begin{equation*}
\sigma_{m}(z):=(dd^{c}\|z\|^{2})^{m-1},\quad
 \eta_{m}(z):=d^{c}\log\|z\|^{2}\wedge(dd^{c}\|z\|^{2})^{m-1}\end{equation*}
  for $z\in\mathbb{C}^{m}\setminus{\{0}\}.$\par

For a divisor $\nu$ on $\mathbb{C}^{m}$ we define the following counting functions of $\nu$ by
\begin{eqnarray*} n(t)=\left\{
                                    \begin{array}{ll}
                                      \int_{|\nu|\cap B(t)}\nu(z)\sigma_{m}(z), & \hbox{if $m\geq 2;$} \\
                                      \sum_{|z|\leq t}\nu(z), & \hbox{if $m=1,$}
                                    \end{array}
                                  \right.
\end{eqnarray*} and
\begin{equation*}
N(r, \nu)=\int_{1}^{r}\frac{n(t)}{t^{2m-1}}dt\quad
(1<r<\infty).\end{equation*}

Let $\varphi(\not\equiv 0)$ be an entire holomorphic function on $\mathbb{C}^{m}.$ For $a\in\mathbb{C}^{m},$ we write $\varphi(z)=\sum_{i=0}^{\infty}P_{i}(z-a),$ where the term $P_{i}$ is a homogeneous polynomial of degree~$i.$ We denote the zero-multiplicity of $\varphi$ at $a$ by $\nu_{\varphi}(a)=\min{\{i:P_{i}\not\equiv 0}\}.$ Thus we can define a divisor $\nu_{\varphi}$ such that $\nu_{\varphi}(z)$ equals the zero multiplicity of $\varphi$ at $z$ in the sense of \cite[Definition 2.1]{fujimoto:74} whenever $z$ is a regular point of an analytic set  $|\nu_{\varphi}|:=\overline{\{z\in \mathbb{C}^{m}: \nu_{\varphi}(z)\neq 0\}}.$ \par

Letting $h$ be a nonzero meromorphic function on $\mathbb{C}^{m}$ with $h=\frac{h_{0}}{h_{1}}$ on $\mathbb{C}^{m}$ and $\dim(h_{0}^{-1}(0)\cap h_{1}^{-1}(0))\leq m-2,$ we define $\nu_{h}^{0}:=\nu_{h_{0}},\nu_{h}^{\infty}:=\nu_{h_{1}}.$\par

For a meromorphic function $h$ on $\mathbb{C}^{m},$ we have the Jensen's theorem:
$$N(r, \nu_{h}^{0})-N(r, \nu_{h}^{\infty})=\int_{S_{m}(r)}\log|h|\eta_{m}(z)-\int_{S_{m}(1)}\log|h|\eta_{m}(z).$$
\par

\noindent\textbf{2.2.} A meromorphic mapping $f:\mathbb{C}^{m}\rightarrow\mathbb{P}^{n}(\mathbb{C})$ is a holomorphic mapping from $U$ into $\mathbb{P}^{n}(\mathbb{C}),$ where $U$  can be chosen so that $V\equiv \mathbb{C}^{m}\setminus U$ is an analytic subvariety of $\mathbb{C}^{m}$ of codimension at least $2.$ Furthermore $f$ can be represented by a holomorphic mapping of $\mathbb{C}^{m}$ to $\mathbb{C}^{n+1}$ such that $$V=I(f)=\{z\in\mathbb{C}^{m}: f_{0}(z)=\cdots=f_{n}(z)=0\},$$ where $f_{0}, \ldots, f_{n}$ are holomorphic functions on $\mathbb{C}^{m}.$ We say that $f=[f_{0}, \ldots, f_{n}]$ is a reduced representation of $f$ (the only factors common to $f_{0}, \ldots, f_{n}$ are units).  If $g=hf$ for $h$ any quotient of holomorphic functions on $\mathbb{C}^{m},$ then $g$ will be called a representation of $F$ (e.g. reduced iff $h$ is holomorphic and a unit). Set $\|f\|=(\sum_{j=0}^{n}|f_{j}|^{2})^{\frac{1}{2}}.$ The growth of meromorphic mapping $f$ is measured by its characteristic function
\begin{eqnarray*}
T_{f}(r)&=&\int_{r_{0}}^{r}\frac{dt}{t^{2m-1}}\int_{B_{m}(t)}dd^{c}\log\|f\|^{2}\wedge \sigma_{m}(z)\\
&=&\int_{S_{m}(r)}\log\|f\|\eta_{m}(z)-
 \int_{S_{m}(1)}\log\|f\|\eta_{m}(z)\\
 &=&\int_{S_{m}(r)}\log\max\{|f_{0}|, \ldots, |f_{n}|\}\eta_{m}(z)+O(1)\quad(r>r_{0}>1).\end{eqnarray*}
Note that $T_{f}(r)$ is independent of the choice of the reduced representation of $f.$ The order and hyper-order of $f$ are respectively defined by
$$\zeta(f):=\limsup_{r\rightarrow\infty}\frac{\log^{+} T_{f}(r)}{\log r}\quad\mbox{and}\quad \zeta_{2}(f):=\limsup_{r\rightarrow\infty}\frac{\log^{+}\log^{+} T_{f}(r)}{\log r},$$ where $\log^{+}x:=\max\{\log x, 0\}$ for any $x>0.$\par

 We say that a meromorphic mapping $f$ from $\mathbb{C}^{m}$ into $\mathbb{P}^{n}(\mathbb{C})$ with a reduced representation $[f_{0},\ldots, f_{n}]$ is linearly nondegenerate over $\mathcal{P}^{\lambda}_{c}$  if the entire functions $f_{0},$ $\ldots,$ $f_{n}$ are linearly independent over $\mathcal{P}^{\lambda}_{c},$ and say that $f$ is linearly nondegenerate over $\mathbb{C}^{1}$  if the entire functions $f_{0},$ $\ldots,$ $f_{n}$ are linearly independent over $\mathbb{C}^{1}.$ \par

\medskip

\noindent\textbf{2.3.} Let hyperplanes $H_{j}$ of $\mathbb{P}^{n}(\mathbb{C})$ be defined by
$$H_{j}:\quad h_{j0}w_{0}+\ldots+h_{jn}w_{n}=0\quad (1\leq j\leq q),$$ where $[w_{0}, \ldots, w_{n}]$ is a homogeneous coordinate system of $\mathbb{P}^{n}(\mathbb{C}).$  Suppose that $[f_{0}, \ldots, f_{n}]$ is a reduced representation of a meromorphic mapping $f:\mathbb{C}^{m}\rightarrow\mathbb{P}^{n}(\mathbb{C}),$ then we denote
$$(f, H_{j})=h_{j0}f_{0}+\ldots+h_{jn}f_{n}$$
which are entire functions on $\mathbb{C}^{m}$ for all $j\in\{1,\ldots, q\}.$

We say that $q$ hyperplanes $H_{j}$ $(1\leq j\leq q)$ are in $N$-subgeneral position of $\mathbb{P}^{n}(\mathbb{C})$ if
$$\bigcap_{j\in R}H_{j}=\emptyset$$ for any subset $R\subset Q=\{1, 2, \ldots, q\}$ with its cardinality $|R|=N+1\geq n+1.$ This is equivalent to that for an arbitrary $(N+1, n+1)$-matrix $(h_{jk})_{j\in R, 0\leq k\leq n},$
$$rank(h_{jk})_{j\in R, 0\leq k\leq n}=n+1.$$ If $H_{j}$ $(1\leq j\leq q)$ are in $n$-subgeneral position, we simply say that they are in general position.\par

We denote by $V(R)$ the vector subspace spanned by  $(h_{jk}w_{k})_{0\leq k\leq n},$ $j\in R\subset Q$ in $\mathbb{C}^{n+1},$ and $rk(R):=\dim V(R),$ $rk(\emptyset)=0.$\par

\medskip

\noindent\textbf{2.4.} Let a meromorphic mapping $f=[f_{0}, \ldots, f_{n}]$  from $\mathbb{C}^{m}$ into $\mathbb{P}^{n}(\mathbb{C})$ and a hyperplane $H$ of $\mathbb{P}^{n}(\mathbb{C})$ satisfy $(f, H)\not\equiv 0.$ The closeness of the image of a meromorphic mapping $f$ to intersecting $H$ is measured by the proximity function
$$m_{f, H}(r)=\int_{S_{m}(r)}\log^{+}\frac{\|f\|\cdot\|H\|}{|(f,H)|}\eta_{m}(z)-\int_{S_{m}(1)}\log^{+}\frac{\|f\|\cdot\|H\|}{|(f, H)|}\eta_{m}(z).$$
We have the first main theorem of Nevanlinna theory
$$T_{f}(r)=N(r, \nu_{(f, H)}^{0})+m_{f, H}(r)+O(1)\quad (r>1).$$

\medskip

\noindent\textbf{2.5.} Let $f$ be a meromorphic mapping from $\mathbb{C}^{m}$ into $\mathbb{P}^{n}(\mathbb{C}).$ For $c=(c_{1}, \ldots, c_{m})$ and  $z=(z_{1}, \ldots, z_{m}),$ we write $c+z=(c_{1}+z_{1},\ldots, c_{m}+z_{m}),$ $cz=(c_{1}z_{1},\ldots, c_{m}z_{m}).$ Denote the $c$-difference operator by
$$\Delta_{c}f(z):=f(c+z)-f(z).$$
We use the short notations
$$f(z)\equiv f:=\overline{f}^{[0]},\,  f(z+c)\equiv \overline{f}:=\overline{f}^{[1]},\,  f(z+2c)\equiv \overline{\overline{f}}\equiv \overline{f}^{[2]}, \ldots, f(z+kc)\equiv \overline{f}^{[k]}.$$
\par

Assume that $f$ has a reduced representation $[f_{0}, \ldots, f_{n}].$ Let
$$D^{(j)}=\left(\frac{\partial}{\partial z_{1}}\right)^{\alpha_{1}(j)} \cdots\left(\frac{\partial}{\partial z_{m}}\right)^{\alpha_{m}(j)}$$
be a partial differentiation operator of order at most $j=\sum_{k=1}^{m}\alpha_{k}(j).$ Similarly as the Wronskian determinant
$$W(f)=W(f_{0}, \ldots, f_{n})=\left|\begin{array}{cccc}
                                 f_{0} & f_{1} & \cdots & f_{n} \\
                                 D^{(1)}f_{0} & D^{(1)}f_{1} & \cdots & D^{(1)}f_{n} \\
                                 \vdots & \vdots & \ddots & \vdots \\
                                 D^{(n)}f_{0} & D^{(n)}f_{1} & \cdots & D^{(n)}f_{n}
                               \end{array}\right|,
$$ the Casorati determinant is defined by $$C(f)=C(f_{0}, \ldots, f_{n})=\left|\begin{array}{cccc}
                                 f_{0} & f_{1} & \cdots & f_{n} \\
                                 \overline{f}_{0} & \overline{f}_{1} & \cdots & \overline{f}_{n} \\
                                 \vdots & \vdots & \ddots & \vdots \\
                                 \overline{f}_{0}^{[n]} & \overline{f}_{1}^{[n]} & \cdots & \overline{f}_{n}^{[n]}
                             \end{array}\right|.
$$\par

For a subset $R\subset Q=\{1, \ldots, q\}$ such that $|R|=n+1,$ we denote by $$C(((f, H_{j}), j\in R))$$ the Casorati determinant of $(f, H_{j}), j\in R$ with increasing order of indices. \par

\section{Proof of Theorem \ref{T-3.1}}\label{proofsec}
We recall two lemmas due to Nochka (see \cite{chen:90,fujimoto:93,nochka:83,noguchi:05}) as follows. \par

\begin{lemma}[\cite{chen:90,fujimoto:93,nochka:83,noguchi:05}]\label{L-3.1} Let $H_{j},$ $j\in Q=\{1,2,\ldots,q\}$ be hyperplanes of $\mathbb{P}^{n}(\mathbb{C})$ in $N$-subgeneral position, and assume that $q>2N-n+1.$ Then there are positive rational constants $\omega(j),$ $j\in Q$ satisfying the following:\par
(i) $0<\omega(j)\leq q$ for all $j\in Q.$\par
(ii) Setting $\tilde{\omega}=\max_{j\in Q}\omega(j),$ one gets
$$\sum_{j=1}^{q}\omega(j)=\tilde{\omega}(q-2N+n-1)+n+1.$$\par
(iii) $\frac{n+1}{2N-N+1}\leq \tilde{\omega}\leq\frac{n}{N}.$\par
(iv) For $R\subset Q$ with $0<|R|\leq N+1,$ $\sum_{j\in R}\omega(j)\leq rk(R).$
\end{lemma}

The above $\omega(j)$ and $\tilde{\omega}$ are called the Nochka weights and the Nochka constant, respectively.\par

\begin{lemma}[\cite{chen:90,fujimoto:93,nochka:83,noguchi:05}] \label{L-3.2}
Let $H_{j},$ $j\in Q=\{1,2,\ldots,q\}$, be hyperplanes of $\mathbb{P}^{n}(\mathbb{C})$ in $N$-subgeneral position, and assume that $q>2N-n+1.$ Let $\{\omega(j)\}$ be their Nochka weights.\par

Let $E_{j}\geq 1,$ $j\in Q$ be arbitrarily given numbers. Then for every subset $R\subset Q$ with $0<|R|\leq N+1,$ there are distinct indices $j_{1}, \ldots, j_{rk(R)}\in Q$ such that $rk(\{j_{l}\}^{rk(R)}_{l=1})=rk(R)$ and
$$\prod_{j\in R}E_{j}^{\omega(j)}\leq \prod_{l=1}^{rk(R)}E_{jl}.$$
\end{lemma}

It is known that holomorphic functions $g_{0}, \ldots, g_{n}$ on $\mathbb{C}^{m}$ are linearly dependent over $\mathbb{C}^{m}$ if and only if their Wronskian determinant $W(g_{0},\ldots, g_{n})$ vanishes identically \cite[Prop. 4.5]{fujimoto:85}. It was mentioned in \cite[Remark 2.6]{wonglw:09} without proof that holomorphic functions $g_{0}, \ldots, g_{n}$ on $\mathbb{C}$ are linearly dependent over $\mathcal{P}_{c}$ if and only if their Casorati determinant $C(g_{0},\ldots, g_{n})$ vanishes identically. The proof of this fact can be seen in the proof of \cite[Lemma 3.2]{halburdkt:14} which, in fact, is a more accurate result because it takes into account the growth order of functions. Here we introduce extensions of these results for the case of several complex variables. \par

\begin{lemma}\label{L-3.4} (i) Let $c\in\mathbb{C}^{m}.$
A meromorphic mapping $f:\mathbb{C}^{m}\rightarrow\mathbb{P}^{n}(\mathbb{C})$ with a reduced representation $[f_{0}, \ldots, f_{n}]$ satisfies $C(f_{0}, \ldots, f_{n})\not\equiv 0$ if and only if $f$ is linearly nondegenerate over the field $\mathcal{P}_{c}.$\par

(ii) Let $c\in\mathbb{C}^{m}.$
If a meromorphic mapping $f:\mathbb{C}^{m}\rightarrow\mathbb{P}^{n}(\mathbb{C})$ with a reduced representation $[f_{0}, \ldots, f_{n}]$ satisfies  $\zeta_{2}(f)<\lambda<+\infty,$ then $C(f_{0}, \ldots, f_{n})\not\equiv 0$ if and only if $f$ is linearly nondegenerate over the field $\mathcal{P}_{c}^{\lambda}(\subset\mathcal{P}_{c}).$
\end{lemma}

\begin{proof}
By the definition of the characteristic function of $f$ and using similar discussion as in \cite[Page 47]{goldbergo:08}, it is not difficult to get that for any meromorphic function $g$ on $\mathbb{C}^{m}$ and $c\in \mathbb{C}^{m}$
$$T_{g(z+c)}(r)=O\left(T_{g(z)}(r+||c||)\right).$$
Then considering the above fact and making use of almost the same discussion as in \cite[Lemma 3.2]{halburdkt:14}, one can complete the proof of (ii). To prove (i) it is just not necessary to consider the growth of $f$ in the proof of (ii). We omit the details.
\end{proof}

\begin{lemma}\label{L-3.5} Let $q>2N-n+1,$ $Q=\{1,\ldots, q\}.$ Suppose that $f:\mathbb{C}^{m}\rightarrow\mathbb{P}^{n}(\mathbb{C})$ is a linearly nondegenerate meromorphic mapping over $\mathcal{P}_{c},$ and $H_{j}$ $(j\in Q)$ are hyperplanes of $\mathbb{P}^{n}(\mathbb{C})$ in $N$-subgeneral position. Let $\omega(j),$ $\tilde{\omega}$ be the Nochka weights and Nochka constant of $\{H_{j}\}_{j\in Q}$ respectively. Then we get that
\begin{eqnarray*}
  \|f\|^{\tilde{\omega}(q-2N+n-1)}&\leq& K\cdot\frac{\left(\prod_{t_{j}\in R}|(\overline{f}^{[j]}, H_{t_{j}})|^{\omega(t_{j})}\right)\cdot\left(\prod_{j\in S}|(f, H_{j})|^{\omega(j)}\right)}{|C(f_{0}, \ldots, f_{n})|}\\&&\cdot\frac{|C(((f, H_{j}), j\in R^{o}))|}{|(f, H_{t_{0}})(\overline{f}, H_{t_{1}})\cdots (\overline{f}^{[n]}, H_{t_{n}})|}
\end{eqnarray*}  for an arbitrary $$z\in \mathbb{C}^{m}\setminus\left(\left\{z\in \mathbb{C}^{m}:\left(\prod_{t_{j}\in R}|\overline{g}^{[j]}_{t_{j}}|^{\omega(t_{j})}\right)\cdot\left(\prod_{j\in S}|(f, H_{j})|^{\omega(j)}\right)=0\right\}\cup I(f)\right),$$ where $K$ depends on $\{H_{j}\}_{j\in Q},$ and $R^{o}, R, S$ are some subsets of $Q$ such that
$$R^{o}=\{t_{0}, t_{1}, \ldots, t_{n}\}\subset R=\{t_{0}, t_{1}, \ldots, t_{n}, t_{n+1}, \ldots, t_{N}\}\subset Q\setminus S.$$
\end{lemma}

\begin{proof}
Since the hyperplanes $\{H_{j}\}_{j=1}^{q}$ are in $N$-subgeneral position of $\mathbb{P}^{n}(\mathbb{C}),$ we have $\bigcap_{j\in R}H_{j}=\emptyset$ for any $R\subset Q$ with $|R|=N+1.$ This implies that there exists a subset $S\subset Q$ with $|S|=q-N-1$ such that $\prod_{j\in S}H_{j}(w)\neq 0.$ Let $I(f)=\{z: f_{0}(z)=f_{2}(z)=\cdots=f_{n}(z)=0\}$ with its codimension $\geq 2.$ For arbitrary fixed point $z\in \mathbb{C}^{m}\setminus\left(\cup_{j\in Q}\{z\in \mathbb{C}^{m}: (\overline{f}^{{[k_{j}]}}, H_{j})=0\}\cup I(f)\right)$ (so, $f(z)\in\mathbb{P}^{n}(\mathbb{C})$ and $(\overline{f}^{[k_{j}]}, H_{j})\in\mathbb{C}^{1}$),  there is a positive constant $K_{jk}$ which depends on $H_{j}$ and $k_{j}\in\mathbb{N}\cup\{0\}$ such that
\begin{equation}\label{E-3.1}
  \frac{1}{|K_{jk}|}\leq \frac{|(\overline{f}^{[k_{j}]}, H_{j})|}{\|f(z)\|}\leq |K_{jk}|
\end{equation} for $j\in S.$\par

Below we set $R=Q\setminus S.$ Then we have $|R|=N+1$ and $rk(R)=n+1.$ Then
\begin{equation}\label{E-3.2}
\prod_{j\in S}\left(\frac{(\overline{f}^{[k_{j}]}, H_{j})}{\|f(z)\||K_{jk}|}\right)^{\omega(j)}=\prod_{j\in R}\left(\frac{\|f(z)\||K_{jk}|}{(\overline{f}^{[k_{j}]}, H_{j})}\right)^{\omega(j)}\cdot\frac{\prod_{j\in Q}|(\overline{f}^{[k_{j}]}, H_{j})|^{\omega(j)}}{(\|f(z)\||K_{jk}|)^{\sum_{j\in Q}\omega(j)}}.
\end{equation}

By Lemma \ref{L-3.1} (ii), for $R=Q\setminus S$ we have
\begin{equation}\label{E-3.3}\sum_{j=1}^{q}\omega(j)=\tilde{\omega}(q-2N+n-1)+n+1.\end{equation}

Replacing $E_{j}$ by $\frac{\|f(z)\||K_{jk}|}{|(\overline{f}^{[k_{j}]}, H_{j})|}$ and making use of Lemma \ref{L-3.2}, for $R=Q\setminus S$ there is a subset $R^{o}=\{j_{1}, \ldots, j_{rk(R)}\}\subset R$ such that $|R^{o}|=rk(\{j_{l}\}^{rk(R)}_{l=1})=rk(R)=n+1$ and
\begin{equation}\label{E-3.4}\prod_{j\in R}\left(\frac{\|f(z)\||K_{jk}|}{|(\overline{f}^{[k_{j}]}, H_{j})|}\right)^{\omega(j)}\leq \prod_{j\in R^{o}}\frac{\|f(z)\||K_{jk}|}{|(\overline{f}^{[k_{j}]}, H_{j})|}.\end{equation}

Since $f$ is linearly non-degenerate over $\mathcal{P}_{c},$ by Lemma \ref{L-3.4} we get $C(f_{0}, \ldots, f_{n})\not\equiv 0.$
Since $\{H_{j}\}_{j=1}^{q}$ are in $N$-subgeneral position, there exists a non-singular matrix $B$ depending on $\{H_{j}\}_{j\in R^{o}}$ such that
 \begin{eqnarray*}
&& C(((f, H_{j}), j\in R^{o}))=C\left((f_{0}, f_{1},\ldots, f_{n})B\right)=C(f_{0}, f_{1},\ldots, f_{n})\times \det B.
\end{eqnarray*}
So, $C(((f, H_{j}), j\in R^{o}))\not\equiv 0.$  Hence, there is a positive constant $K_{R^{o}}$ depending on $H_{j}$ such that
\begin{equation}\label{E-3.5}
|K_{R^{o}}|\frac{|C(((f, H_{j}), j\in R^{o}))|}{|C(f_{0}, \ldots, f_{n})|}=1.
\end{equation}

For the above $R^{o}, R, S, Q,$ we may rewrite their elements as follows:
$$Q=\{1,2,\ldots, q\}:=\{t_{0}, t_{1}\ldots, t_{q}\},$$
$$R^{o}=\{t_{0}, t_{1}, \ldots, t_{n}\},\,\, R=\{t_{0}, t_{1}, \ldots, t_{n}, t_{n+1}, \ldots, t_{N}\}.$$ Denote $\overline{g}^{[k]}_{j}:=(\overline{f}^{[k]}, H_{j}),$ $j\in Q.$ Then it follows from \eqref{E-3.1} and \eqref{E-3.2} that for any $z\in G:=\mathbb{C}^{m}\setminus\left(\{z\in \mathbb{C}^{m}:\left(\prod_{t_{j}\in R}|\overline{g}^{[j]}_{t_{j}}|^{\omega(t_{j})}\right)\cdot\left(\prod_{j\in S}|(f, H_{j})|^{\omega(j)}\right)=0\}\cup I(f)\right),$
\begin{eqnarray*}
&&\prod_{j\in S}\left(\frac{1}{|K_{j}|^{2}}\right)^{\omega(j)}\leq\prod_{j\in S}\left(\frac{(f, H_{j})(z)}{\|f(z)\||K_{j}|}\right)^{\omega(j)}\\
&\leq&\prod_{t_j\in R}\left(\frac{\|f(z)\||K_{t_{j}}|}{|\overline{g}^{[j]}_{t_{j}}|}\right)^{\omega(t_{j})}\cdot\frac{\left(\prod_{t_{j}\in R}|\overline{g}^{[j]}_{t_{j}}|^{\omega(t_{j})}\right)\cdot\left(\prod_{j\in S}|(f, H_{j})|^{\omega(j)}\right)}{(\|f(z)\||\min\{K_{1},\ldots, K_{q}\}|)^{\sum_{j\in Q}\omega(j)}}.\end{eqnarray*}
Then together with \eqref{E-3.4},  the above inequality becomes
\begin{eqnarray*}
&&\prod_{j\in S}\left(\frac{1}{|K_{j}|^{2}}\right)^{\omega(j)}\\
&\leq&\prod_{t_{j}\in R^{o}}\frac{\|f(z)\||K_{t_{j}}|}{|\overline{g}^{[j]}_{t_{j}}|}\cdot\frac{\left(\prod_{t_{j}\in R}|\overline{g}^{[j]}_{t_{j}}|^{\omega(t_{j})}\right)\cdot\left(\prod_{j\in S}|(f, H_{j})|^{\omega(j)}\right)}{(\|f(z)\||\min\{K_{1}, \ldots, K_{q}\}|)^{\sum_{j\in Q}\omega(j)}}\end{eqnarray*}for any $z\in G.$
Then by \eqref{E-3.3}, we get from the above inequality that for any $z\in G,$
\begin{eqnarray*}
&&\prod_{j\in S}\left(\frac{1}{|K_{j}|^{2}}\right)^{\omega(j)}\\
&\leq&\prod_{t_{j}\in R^{o}}\frac{\|f(z)\||K_{t_{j}}|}{|\overline{g}^{[j]}_{t_{j}}|}\cdot\frac{\left(\prod_{t_{j}\in R}|\overline{g}^{[j]}_{t_{j}}|^{\omega(t_{j})}\right)\cdot\left(\prod_{j\in S}|(f, H_{j})|^{\omega(j)}\right)}{(\|f(z)\||\min\{K_{1}, \ldots, K_{q}\}|)^{\tilde{\omega}(q-2N+n-1)+n+1}}\\
&=&\frac{\prod_{t_{j}\in R^{o}}|K_{t_{j}}|}{|g_{t_{0}}\overline{g}_{t_{1}}\cdots \overline{g}^{[n]}_{t_{n}}|}\cdot\frac{\left(\prod_{t_{j}\in R}|\overline{g}^{[j]}_{t_{j}}|^{\omega(t_{j})}\right)\cdot\left(\prod_{j\in S}|(f, H_{j})|^{\omega(j)}\right)}{(\|f(z)\||\min\{K_{1}, \ldots, K_{q}\}|)^{\tilde{\omega}(q-2N+n-1)}}.\end{eqnarray*}
By \eqref{E-3.5}, the last line in the above inequalities is equal to
\begin{eqnarray*}
&&\frac{|K_{R^{o}}|\prod_{t_{j}\in R^{o}}|K_{t_{j}}|}{|\min\{K_{1}, \ldots, K_{q}\}|^{\tilde{\omega}(q-2N+n-1)}}\cdot\frac{1}{\|f(z)\|^{\tilde{\omega}(q-2N+n-1)}}\\&&\cdot\frac{\left(\prod_{t_{j}\in R}|\overline{g}^{[j]}_{t_{j}}|^{\omega(t_{j})}\right)\cdot\left(\prod_{j\in S}|(f, H_{j})|^{\omega(j)}\right)}{|C(f_{0}, \ldots, f_{n})|}\cdot\frac{|C(((f, H_{j}), j\in R^{o}))|}{|g_{t_{0}}\overline{g}_{t_{1}}\cdots \overline{g}^{[n]}_{t_{n}}|}.
\end{eqnarray*}
So, we get from the above discussion that for any $z\in G,$
\begin{eqnarray*}
&&\|f(z)\|^{\tilde{\omega}(q-2N+n-1)}\\&\leq& \frac{|K_{R^{o}}|\prod_{t_{j}\in R^{o}}|K_{t_{j}}|\prod_{j\in S}|K_{j}|^{2\omega(j)}}{|\min\{K_{1}, \ldots, K_{q}\}|^{\tilde{\omega}(q-2N+n-1)}}\\&&\cdot\frac{\left(\prod_{t_{j}\in R}|\overline{g}^{[j]}_{t_{j}}|^{\omega(t_{j})}\right)\cdot\left(\prod_{j\in S}|(f, H_{j})|^{\omega(j)}\right)}{|C(f_{0}, \ldots, f_{n})|}\cdot\frac{|C(((f, H_{j}), j\in R^{o}))|}{|g_{t_{0}}\overline{g}_{t_{1}}\cdots \overline{g}^{[n]}_{t_{n}}|}
\end{eqnarray*}
Therefore, the inequality in the assertion of this lemma is obtained immediately by setting $$K=\frac{|K_{R^{o}}|\prod_{t_{j}\in R^{o}}|K_{t_{j}}|\prod_{j\in S}|K_{j}|^{2\omega(j)}}{|\min\{K_{1}, \ldots, K_{q}\}|^{\tilde{\omega}(q-2N+n-1)}}$$ which is a positive constant depending on $\{H_{j}\}_{j\in Q}.$
\end{proof}

The following result is a difference analogue of the lemma on the logarithmic derivative in several complex variables. It generalizes the one dimensional results \cite[Theorem 5.1]{halburdkt:14}, \cite[Theorem 2.1]{halburdk:06AASFM} and the high dimensional result \cite[Theorem 3.1]{korhonen:12}. In \cite{korhonen:12} a difference analogue of the lemma on the logarithmic derivatives was obtained for meromorphic functions in several variables of hyperorder strictly less that $2/3$. The following lemma extends this result for the case hyperorder $<1$. \par

\begin{lemma}\label{L-3.6}
Let $f$ be a nonconstant meromorphic function on $\mathbb{C}^{m}$ such that $f(0)\neq 0,\infty,$ and let $\varepsilon>0.$ If $\zeta_{2}(f):=\zeta<1,$ then
\begin{eqnarray*}
m\left(r, \frac{f(z+c)}{f(z)}\right)=\int_{S_{m}(r)}\log^{+}\left|\frac{f(z+c)}{f(z)}\right|\eta_{m}(z)=o\left(\frac{T_{f}(r)}{r^{1-\zeta-\varepsilon}}\right)
\end{eqnarray*} for all $r>0$ outside of a possible exceptional set $E\subset[1, \infty)$ of finite logarithmic measure $\int_{E}\frac{dt}{t}<\infty.$
\end{lemma}

\begin{proof}
Let $E_{1}$ be the set of all points $\xi\in S_{m}(1)$ such that $\{z=u\xi: |u|<+\infty\}\subset I(f)$ which is of measure zero in $S_{m}(1).$ For any $\xi\in S_{m}(1)\setminus E_{1},$ considering the meromorphic function $f^{\xi}(u):=f(\xi u)$ of $\mathbb{C}^{1},$ we have
 \begin{eqnarray*} T_{f^{\xi}}(r)=\frac{1}{2\pi}\int_{0}^{2\pi}\log |f^{\xi}(re^{i\theta})|d\theta-\log|f(0)|,\end{eqnarray*} and thus by \cite[Lemmas~1.1--1.2]{stoll:64} it follows (see also \cite[pages 33--34]{fujimoto:74Nagoya})
\begin{equation}\label{Eqq-1}T_{f}(r)=\int_{S_{m}(1)}T_{f^{\xi}}(r)\eta_{m}(z).\end{equation}\par

Recall that the proximity function of a meromorphic function $\phi$ on $\mathbb{C}^{m}$ is defined (\cite[Definition 5.5]{fujimoto:74Nagoya}) by
\begin{equation*}
m(r, \phi):=\int_{S_{m}(r)}\log^{+}\left|\phi\right|\eta_{m}(z).
\end{equation*}
Let $E_{2}$ be the set of all points $\xi\in S_{m}(1)$ such that $\{u\xi: |u|<+\infty\}\subset I(\phi)$ which is of measure zero in $S_{m}(1).$ For any $\xi\in S_{m}(1)\setminus E_{2},$ considering the meromorphic function $\phi^{\xi}(u):=\phi(\xi u)$ of $\mathbb{C}^{1},$ we have
$$m(r, \phi^{\xi})=\frac{1}{2\pi}\int_{0}^{2\pi}\log^{+}|\phi^{\xi}(re^{i\theta})|d\theta.$$ Then by \cite[Lemmas~1.1--1.2]{stoll:64}, we also get
\begin{eqnarray}\label{Eqq-2}m(r,\phi)=\int_{S_{m}(1)}m(r,¡¡\phi^{\xi})\eta_{m}(z).\end{eqnarray} \par

Now let the constant $c:=\tilde{c}\xi,$ where $\tilde{c}\in\mathbb{C}^{1}\setminus\{0\}.$ For any $\xi\in S_{m}(1)\setminus E_{1},$ considering the meromorphic function $f^{\xi}(u):=f(\xi u)$ of $\mathbb{C}^{1},$ we get from \eqref{Eqq-2} that
\begin{eqnarray*} m\left(r, \frac{f(z+c)}{f(z)}\right)&=&
  \int_{S_{m}(r)}\log^{+}\left|\frac{f(z+c)}{f(z)}\right|\eta_{m}(z)
  \\&=&\int_{S_{m}(1)}\left(\frac{1}{2\pi}\int_{0}^{2\pi}\log^{+}\left|\frac{f^{\xi}(re^{i\theta}+\tilde{c})}{f^{\xi}(re^{i\theta})}\right|d\theta\right)\eta_{m}(z),\end{eqnarray*}
where we denote $z=u\xi$ for any $\xi\in S_{m}(1).$ By \cite[Lemma 8.2]{halburdkt:14}, we get that for all $r>0,$ $\delta\in(0,1)$ and $\alpha>1,$ \begin{eqnarray*}
m\left(r,\frac{f^{\xi}(re^{i\theta}+\tilde{c})}{f^{\xi}(re^{i\theta})}\right)&=&\frac{1}{2\pi}\int_{0}^{2\pi}\log^{+}\left|\frac{f^{\xi}(re^{i\theta}+\tilde{c})}{f^{\xi}(re^{i\theta})}\right|d\theta\\&\leq& \frac{K(\alpha, \delta, \tilde{c})}{r^{\delta}}\left(T_{f^{\xi}}(\alpha (r+|\tilde{c}|))+\log^{+}\frac{1}{|f^{\xi}(0)|}\right),\end{eqnarray*}
where $K(\alpha, \delta, \tilde{c})=\frac{4|\tilde{c}|^{\delta}(4\alpha+\alpha \delta+\delta)}{\delta(1-\delta)(\alpha-1)},$ $r=|u|=\|u\xi\|=\|z\|.$ Therefore, together with \eqref{Eqq-1}, it follows from the two inequalities above that
\begin{eqnarray*} m\left(r, \frac{f(z+c)}{f(z)}\right)&=&
  \int_{S_{m}(r)}\log^{+}\left|\frac{f(z+c)}{f(z)}\right|\eta_{m}(z)
  \\&\leq&\int_{S_{m}(1)}\left(\frac{K(\alpha, \delta, \tilde{c})}{r^{\delta}}\left(T_{f^{\xi}}(\alpha (r+|\tilde{c}|))+\log^{+}\frac{1}{|f^{\xi}(0)|}\right)\right)\eta_{m}(z)\\
  &=&
\frac{K(\alpha, \delta, \tilde{c})}{r^{\delta}} \int_{S_{m}(1)}T_{f^{\xi}}(\alpha (r+|\tilde{c}|))\eta_{m}(z)+O(1),\end{eqnarray*}
namely, \begin{eqnarray}\label{Eqq-3} m\left(r, \frac{f(z+c)}{f(z)}\right)\leq\frac{K(\alpha, \delta, \tilde{c})}{r^{\delta}}T_{f}(\alpha (r+|\tilde{c}|))+O(1).\end{eqnarray}

The following part of the proof is dealt with similarly as in \cite[Theorem 5.1]{halburdkt:14}.  Choose $p(r):=r,$ $h(x):=(\log x)^{1+\frac{\varepsilon}{3}}$ and $$\alpha:=1+\frac{p(r+|\widetilde{c}|)}{(r+|\widetilde{c}|)h(T_{f}(r+|\widetilde{c}|))},$$ and thus $$\rho=\alpha(r+|\widetilde{c}|)=r+|\widetilde{c}|+\frac{r+|\widetilde{c}|}{(\log T_{f}(r+|\widetilde{c}|))^{1+\frac{\varepsilon}{3}}}.$$ By \cite[Lemma 4]{hinkkanen:92} we have
\begin{eqnarray}\label{Eqq-4}T_{f}(\rho)=T_{f}\left(s+\frac{p(s)}{h(T_{f}(s))}\right)\leq K T_{f}(s)\quad(s=r+|\widetilde{c}|)\end{eqnarray} for all $s$ outside of a set $E$ satisfying \begin{equation*}\int_{E\cap[s_{0}, R]}\frac{ds}{p(s)}\leq\frac{1}{\log K}\int_{e}^{T_{f}(R)}\frac{dx}{xh(x)}+O(1)<\infty\end{equation*} where $R<+\infty$ and $K$ is a positive real constant. Since $\varsigma=\varsigma_{2}(f)<1,$ by \cite[Lemma 8.3]{halburdkt:14} we have \begin{eqnarray}\label{Eqq-5} T_{f}(r+|\widetilde{c}|)=T_{f}(r)+o\left(\frac{T_{f}(r)}{r^{1-\varsigma-\varepsilon}}\right)\end{eqnarray}
for  all $r>0$ outside of a possible exceptional set $F\subset[1, \infty)$ of finite logarithmic measure $\int_{F}\frac{dt}{t}<\infty.$
We can choose suitable $\delta\in(0,1)$ such that  $$\frac{T_{f}(s)}{r^{\delta}}=o\left(\frac{T_{f}(r+|\widetilde{c}|)}{r^{1-\varsigma-\varepsilon}}\right)$$  for all $r\not\in F\cup E.$ Hence it follows from \eqref{Eqq-3}, \eqref{Eqq-4} and \eqref{Eqq-5} that
\begin{eqnarray*}m\left(r, \frac{f(z+c)}{f(z)}\right)=o\left(\frac{T_{f}(r)}{r^{1-\varsigma-\varepsilon}}\right)\end{eqnarray*} for  all $r>0$ outside of a possible exceptional set, still say $E\subset[1, \infty),$ of finite logarithmic measure $\int_{E}\frac{dt}{t}<\infty.$
\end{proof}

 Since $$\frac{f(z)}{f(z+c)}=\frac{f[(z+c)-c]}{f(z+c)},\quad \frac{\overline{f}^{[k]}}{f(z)}=\frac{\overline{f}^{[k]}}{\overline{f}^{[k-1]}}\cdot\frac{\overline{f}^{[k-1]}}{\overline{f}^{[k-2]}}\cdots\frac{\overline{f}}{f(z)}\, \,(k\in\mathbb{N}),$$ it follows immediately from Lemma \ref{L-3.6} that
\begin{eqnarray*}
\int_{S_{m}(r)}\log^{+}\left|\frac{f(z)}{f(z+c)}\right|\eta_{m}(z)&=&o\left(\frac{T_{f}(r)}{r^{1-\zeta-\varepsilon}}\right),\\
 \int_{S_{m}(r)}\log^{+}\left|\frac{\overline{f}^{[k]}}{f(z)}\right|\eta_{m}(z)&=&o\left(\frac{T_{f}(r)}{r^{1-\zeta-\varepsilon}}\right)
\end{eqnarray*} for all $r>0$ outside of a possible exceptional set $E\subset[1, \infty)$ of finite logarithmic measure $\int_{E}\frac{dt}{t}<\infty.$

\begin{proof}[Proof of Theorem \ref{T-3.1}]
 Assume that $q>2N-n+1.$ Let $Q=\{1,,2\ldots, q\}.$ By Lemma \ref{L-3.5}, for $r>1$ we have
\begin{eqnarray*}
  &&\tilde{\omega}(q-2N+n-1)\log\|f\|\\&\leq& \sum_{t_{j}\in R}\omega(t_{j})\log|(\overline{f}^{[j]}, H_{t_{j}})|+
  \sum_{j\in S}\omega(j)\log|(f, H_{j})|-\log|C(f_{0}, f_{1}, \ldots, f_{n})|\\&&
   +\log\frac{C(((f, H_{j}), j\in R^{o}))}{|(f, H_{t_{0}})(\overline{f}, H_{t_{1}})\cdots (\overline{f}^{[n]}, H_{t_{n}})|}+O(1)
\end{eqnarray*} for some subsets $R^{o}, R, S$ of $Q$ such that
$$R^{o}=\{t_{0}, t_{1}, \ldots, t_{n}\}\subset R=\{t_{0}, t_{1}, \ldots, t_{n}, t_{n+1}, \ldots, t_{N}\}\subset Q\setminus S.$$
Integrating both sides of this inequality, we have
\begin{eqnarray*}
 && \tilde{\omega}(q-2N+n-1)\int_{S_{m}(r)}\log\|f\|\sigma_{m}(z)\\
 &\leq& \sum_{t_{j}\in R}\omega(t_{j})\int_{S_{m}(r)}\log|(\overline{f}^{[j]}, H_{t_{j}})|\eta_{m}(z)
 +
 \sum_{j\in S}\omega(j)\int_{S_{m}(r)}\log|(f, H_{j})|\eta_{m}(z)\\&&
 -\int_{S_{m}(r)}\log|C(f_{0}, f_{1}, \ldots, f_{n})|\eta_{m}(z)\\&&
   +\int_{S_{m}(r)}\log^{+}\frac{C(((f, H_{j}), j\in R^{o}))}{|(f, H_{t_{0}})(\overline{f}, H_{t_{1}})\cdots (\overline{f}^{[n]}, H_{t_{n}})|}\eta_{m}(z)+O(1).
\end{eqnarray*}
By the definition of the characteristic function of $f$ and together with the Jensen's theorem,
\begin{eqnarray*}
&& \tilde{\omega}(q-2N+n-1)T_{f}(r)\\
 &\leq& \sum_{t_{j}\in R}\omega(t_{j})N(r, \nu_{(\overline{f}^{[j]}, H_{t_{j}})}^{0})+\sum_{j\in S}\omega(j)N(r, \nu_{(f, H_{j})}^{0})-N(r, \nu_{C(f_{0}, f_{1}, \ldots, f_{n})}^{0})\\
  && +\int_{S_{m}(r)}\log^{+}\frac{C(((f, H_{j}), j\in R^{o}))}{|(f, H_{t_{0}})(\overline{f}, H_{t_{1}})\cdots (\overline{f}^{[n]}, H_{t_{n}})|}\eta_{m}(z)+O(1)\\
   &\leq& \sum_{t_{j}\in R}\omega(t_{j})N(r+j|c|, \nu_{(f, H_{t_{j}})}^{0})+\sum_{j\in S}\omega(j)N(r, \nu_{(f, H_{j})}^{0})-N(r, \nu_{C(f_{0}, f_{1}, \ldots, f_{n})}^{0})\\
  && +\int_{S_{m}(r)}\log^{+}\frac{C(((f, H_{j}), j\in R^{o}))}{|(f, H_{t_{0}})(\overline{f}, H_{t_{1}})\cdots (\overline{f}^{[n]}, H_{t_{n}})|}\eta_{m}(z)+O(1).\end{eqnarray*}
By the Jensen's theorem and the definition of characteristic function, we have
\begin{eqnarray*}
  N(r, \nu_{(f, H_{t_j})}^{0})&=&\int_{S_{m}(r)}\log|(f, H_{t_j})|\eta_{m}(z)+O(1)\\
  &\leq& \int_{S_{m}(r)}\log\|f\|\eta_{m}(z)+O(1)
  =T_{f}(r)+O(1).
\end{eqnarray*} Thus the hyperorder of $N(r, \nu_{(f, H_{j})}^{0})$ satisfies
 \begin{eqnarray*}\lambda_{t_j}:=\limsup_{r\rightarrow\infty} \frac{\log\log N(r, \nu_{(f, H_{t_{j}})}^{0})}{\log r}\leq \zeta_{2}(f):=\zeta<1.
\end{eqnarray*} Then by \cite[Lemma 8.3]{halburdkt:14} we obtain
 \begin{eqnarray*}
 N(r+j|c|, \nu_{(f, H_{t_{j}})}^{0})&\leq& N(r, \nu_{(f, H_{t_j})}^{0})+o\left(\frac{ N(r, \nu_{(f, H_{t_j})}^{0})}{r^{1-\lambda_{t_j}-\varepsilon}}\right)
 \\&\leq& N(r, \nu_{(f, H_{j})}^{0})+o\left(\frac{T_{f}(r)}{r^{1-\zeta-\varepsilon}}\right).\end{eqnarray*}
So, it follows that
  \begin{eqnarray*}
  && \tilde{\omega}(q-2N+n-1)T_{f}(r)\\
  &\leq&\sum_{j\in R}\omega(j)N(r, \nu_{(f, H_{j})}^{0})+\sum_{j\in S}\omega(j)N(r, \nu_{(f, H_{j})}^{0})-N(r, \nu_{C(f_{0}, f_{1}, \ldots, f_{n})}^{0})\\
  && +\int_{S_{m}(r)}\log^{+}\frac{C(((f, H_{j}), j\in R^{o}))}{|(f, H_{t_{0}})(\overline{f}, H_{t_{1}})\cdots (\overline{f}^{[n]}, H_{t_{n}})|}\eta_{m}(z)+o\left(\frac{T_{f}(r)}{r^{1-\zeta-\varepsilon}}\right)\\
   &\leq& \sum_{j\in Q}\omega(j)N(r, \nu_{(f, H_{j})}^{0})-N(r, \nu_{C(f_{0}, f_{1}, \ldots, f_{n})}^{0})\\
  && +\int_{S_{m}(r)}\log^{+}\frac{C(((f, H_{j}), j\in R^{o}))}{|(f, H_{t_{0}})(\overline{f}, H_{t_{1}})\cdots (\overline{f}^{[n]}, H_{t_{n}})|}\eta_{m}(z)+o\left(\frac{T_{f}(r)}{r^{1-\zeta-\varepsilon}}\right).
\end{eqnarray*}
Denote $\overline{g}^{[j]}_{t_{j}}:=(\overline{f}^{[j]}, H_{t_{j}}),$ $t_{j}\in R^{o}.$ We have
\begin{eqnarray*}
\frac{C(((f, H_{j}), j\in R^{o}))}{|(f, H_{t_{0}})(\overline{f}, H_{t_{1}})\cdots (\overline{f}^{[n]}, H_{t_{n}})|}
&=&\frac{\left|\begin{array}{cccc}
                                 g_{t_{0}} & g_{t_{1}} & \cdots & g_{t_{n}} \\
                                 \overline{g}_{t_{0}} & \overline{g}_{t_{1}} & \cdots & \overline{g}_{t_{n}} \\
                                 \vdots & \vdots & \ddots & \vdots \\
                                 \overline{g}_{t_{0}}^{[n]} & \overline{g}_{t_{1}}^{[n]} & \cdots & \overline{g}_{t_{n}}^{[n]}
                             \end{array}\right|}{|g_{t_{0}}\overline{g}_{t_{1}}\cdots \overline{g}^{[n]}_{t_{n}}|}\\
                             &=&\frac{\left|\begin{array}{cccc}
                                 1 & \frac{g_{t_{1}}}{g_{t_{0}}} & \cdots & \frac{g_{t_{n}}}{g_{t_{0}}} \\
                                1 & \frac{\overline{g}_{t_{1}}}{\overline{g}_{t_{0}}} & \cdots & \frac{\overline{g}_{t_{n}}}{\overline{g}_{t_{0}}} \\
                                 \vdots & \vdots & \ddots & \vdots \\
                                 1 & \frac{\overline{g}_{t_{1}}^{[n]}}{\overline{g}_{t_{0}}^{[n]}} & \cdots & \frac{\overline{g}_{t_{n}}^{[n]}}{\overline{g}_{t_{0}}^{[n]}}
                             \end{array}\right|}{|\frac{\overline{g}_{t_{1}}}{{\overline{g}_{t_{0}}}}\cdots \frac{\overline{g}^{[n]}_{t_{n}}}{\overline{g}^{n}_{t_{0}}}|}\\
                             &=&\frac{\left|\begin{array}{cccc}
                                 1& 1& \cdots & 1 \\
                                1 & \left(\frac{\overline{g}_{t_{1}}}{\overline{g}_{t_{0}}}\right)/\left(\frac{g_{t_{1}}}{g_{t_{0}}}\right) & \cdots & \left(\frac{\overline{g}_{t_{n}}}{\overline{g}_{t_{0}}}\right)/\left(\frac{g_{t_{n}}}{g_{t_{0}}}\right)\\
                                 \vdots & \vdots & \ddots & \vdots \\
                                1 & \left(\frac{\overline{g}^{[n]}_{t_{1}}}{\overline{g}^{[n]}_{t_{0}}}\right)/\left(\frac{g_{t_{1}}}{g_{t_{0}}}\right) & \cdots & \left(\frac{\overline{g}^{[n]}_{t_{n}}}{\overline{g}^{[n]}_{t_{0}}}\right)/\left(\frac{g_{t_{n}}}{g_{t_{0}}}\right)
                             \end{array}\right| }{
                             \left|\left[\left(\frac{\overline{g}_{t_{1}}}{\overline{g}_{t_{0}}}\right)/\left(\frac{g_{t_{1}}}{g_{t_{0}}}\right)\right]\cdots \left[ \left(\left(\frac{\overline{g}^{[n]}_{t_{n}}}{\overline{g}^{[n]}_{t_{0}}}\right)/\frac{g_{t_{n}}}{g_{t_{0}}}\right)\right]\right|} .\\
\end{eqnarray*}
By the definition of the characteristic function, one can deduce (or by \cite{fujimoto:98,ru:01}), for $i\neq j$
\begin{equation*} T_{\frac{(f,H_{i})}{(f, H_{j})}}(r)\leq T_{f}(r)+O(1),\end{equation*}
and thus $\zeta_{2}(\frac{(f,H_{i})}{(f, H_{j})})\leq \zeta_{2}(f):=\zeta<1.$ Hence by
Lemma \ref{L-3.6}  we have

\begin{eqnarray*}
&&\int_{S_{m}(r)}\log^{+}\frac{C(((f, H_{j}), j\in R^{o}))}{|(f, H_{t_{0}})(\overline{f}, H_{t_{1}})\cdots (\overline{f}^{[n]}, H_{t_{n}})|}\eta_{m}(z)\\
&&=\sum_{j=1}^{n}o\left(\frac{T_{\frac{g_{t_{j}}}{g_{t_{0}}}}(r)}{r^{1-\zeta_{2}(\frac{g_{t_{j}}}{g_{t_{0}}})-\varepsilon}}\right)\\
&&= o\left(\frac{T_{f}(r)}{r^{1-\zeta-\varepsilon}}\right)\end{eqnarray*}for all $r>0$ outside of a possible exceptional set $E\subset[1, \infty)$ of finite logarithmic measure $\int_{E}\frac{dt}{t}<\infty.$\par

Therefore, the above inequalities implies that
\begin{eqnarray*}
&&(q-2N+n-1)T_{f}(r)\\
&& \leq \sum_{j\in Q}\frac{\omega(j)}{\tilde{\omega}}N(r, \nu_{(f, H_{j})}^{0})-\frac{1}{\tilde{\omega}}N(r, \nu_{C(f_{0}, f_{1}, \ldots, f_{n})}^{0})+o\left(\frac{T_{f}(r)}{r^{1-\zeta-\varepsilon}}\right)
 \end{eqnarray*}
 for all $r>0$ outside of a possible exceptional set $E\subset[1, \infty)$ of finite logarithmic measure $\int_{E}\frac{dt}{t}<\infty.$

From Lemma \ref{L-3.1}, $\tilde{\omega}=\max_{j\in Q}\{\omega(j)\}\leq \frac{n}{N}.$ Then it follows that
 \begin{eqnarray*}
&&\| (q-2N+n-1)T_{f}(r)\\
 &&\leq\sum_{j\in Q}N(r, \nu_{(f, H_{j})}^{0})-\frac{N}{n}N(r, \nu_{C(f_{0}, f_{1}, \ldots, f_{n})}^{0})+o\left(\frac{T_{f}(r)}{r^{1-\zeta-\varepsilon}}\right).
 \end{eqnarray*}Thus Theorem \ref{T-3.1} is proved.\end{proof}

\section{Defect relation}\label{defectsec}
The defects $\delta(f, H)$ and $\delta_{W(f)}$ of a meromorphic mapping $f:\mathbb{C}^{m}\rightarrow\mathbb{P}^{n}(\mathbb{C})$ for a  hyperplane $H$ in $\mathbb{P}^{n}(\mathbb{C})$ are defined by

\begin{eqnarray*}
\delta(f, H)&=&1-\limsup_{r\rightarrow\infty}\frac{N(r, \nu_{(f, H)}^{0})}{T_{f}(r)},\\
\delta_{W(f)}&=&\liminf_{r\rightarrow\infty}\frac{N(r, \nu_{W(f_{0}, \ldots, f_{n})}^{0})}{T_{f}(r)}.
\end{eqnarray*}
From the Chen's version of the second main theorem (Theorem \ref{T-1.0}), there exists a defect relation such that
$$\frac{N+1}{n+1}\delta_{W(f)}+\sum_{j=1}^{q}\delta(f, H)\leq 2N-n+1$$ for $q$ hyperplanes $\{H_{j}\}_{j=1}^{q}.$\par

Similary, the difference defect $\delta_{C(f)}$ of a meromorphic mapping $f:\mathbb{C}^{m}\rightarrow\mathbb{P}^{n}(\mathbb{C})$ with reduced representation $[f_{0}, \ldots, f_{n}]$ is defined by
\begin{eqnarray*}
\delta_{C(f)}=\liminf_{r\rightarrow\infty}\frac{N(r, \nu_{C(f_{0}, \ldots, f_{n})}^{0})}{T_{f}(r)}.
\end{eqnarray*}
Hence, by Theorem \ref{T-3.1} we obtain a defect relation as follows, which is an extension of \cite[Corollary 3.4]{wonglw:09}.\par

\begin{theorem}Under the conditions of Theorem~\ref{T-3.1} we have the defect relation
$$\frac{N}{n}\delta_{C(f)}+\sum_{j=1}^{q}\delta(f, H)\leq 2N-n+1.$$
\end{theorem}

\section{Uniqueness of meromorphic mappings}\label{uniquenesssec}

The uniqueness problem for meromorphic mappings under some conditions on the inverse images of divisors was first investigated by R. Nevanlinna. He \cite{nevanlinna:26} proved that if two nonconstant meromorphic functions $f$ and $g$ on the complex plane $\mathbb{C}^{1}$ have the same inverse images ignoring multiplicities for five distinct values in $\mathbb{P}^{1}(\mathbb{C}),$ then $f\equiv g.$ In 1975, H. Fujimoto \cite{fujimoto:75} generalized Nevanlinna's five-value theorem to the case of higher dimension by showing that if two linearly nondegenerate meromorphic mappings $f, g: \mathbb{C}^{m}\rightarrow\mathbb{P}^{n}(\mathbb{C})$ have the same inverse images counted with multiplicities for $q\geq 3n+2$ hyperplanes in general position in $\mathbb{P}^{n}(\mathbb{C}),$ then $f\equiv g.$ For basic results in the uniqueness theory of meromorphic functions and mappings, we refer to two books \cite{yangy:03,huly:03}.\par

By considering the uniqueness problem for $f(z)$ and $f(z+c)$ intersecting hyperplanes in $N$-subgeneral position, we obtain the following uniqueness theorem.\par

\begin{theorem}\label{T-5.3} Let $f$ be a meromorphic mapping  with hyper-order $\varsigma(f)<1$ from $\mathbb{C}^{m}$ into $\mathbb{P}^{n}(C),$ and let $\tau(z)=z+c,$ where $c\in \mathbb{C}^{m}.$ If $\tau((f, H_{j})^{-1})\subset (f,  H_{j})^{-1}$ (counting multiplicity) hold for $n+p$ distinct hyperplanes $\{H_{j}\}_{j=1}^{n+p}$ in $N$-subgeneral position in $\mathbb{P}^{n}(\mathbb{C}),$ and if $p>\frac{N}{N-n+1}+N-n,$ then $f(z)=f(z+c).$
\end{theorem}

We say that the pre-image of $(f, H)$ for a meromorphic mapping $f: \mathbb{C}^{m}\rightarrow \mathbb{P}^{n}(\mathbb{C})$ intersecting a hyperplane $H$ of $\mathbb{P}^{n}(\mathbb{C})$ is forward invariant with respect to the translation $\tau=z+c$ if $\tau((f, a)^{-1})\subset (f, a)^{-1}$  where  $\tau((f, a)^{-1})$ and $(f, a)^{-1}$  are considered to be multi-sets in which each point is repeated according to its multiplicity. By this definition the (empty and thus forward invariant) pre-images of the usual Picard exceptional values become special cases of forward invariant pre-images. Then Theorem \ref{T-5.3} is an extension of the  Picard's theorem under the growth condition "hyperorder $<1$". Actually, Theorem \ref{T-5.3} is proved from a generalized Picard-type theorem which will be shown in the next section.\par

\section{Difference analogue of a generalized Picard-type theorem}\label{picardsec}

Fujimoto \cite{fujimoto:72b} and Green \cite{green:72} gave a natural generalization of the Picard's theorem by showing that if $f: \mathbb{C}\rightarrow\mathbb{P}^{n}(\mathbb{C})$ omits $n+p$ hyperplanes in general position where $p\in\{1, \ldots, n+1\},$ then the image of $f$ is contained in a linear subspace of dimension at most $[\frac{n}{p}].$ Recently, Halburd, Korhonen and Tohge \cite{halburdkt:14} proposed a difference analogue of the general Picard-type theorem for homomorphic curves with hyperorder strictly less than one.\par

 \begin{theorem}[\cite{halburdkt:14}] \label{T-6.1} Let $f: \mathbb{C}\rightarrow\mathbb{P}^{n}(\mathbb{C})$ be a holomorphic curve such that hyperorder $\zeta_{2}(f)<1,$ let $c\in \mathbb{C},$ and let $p\in \{1, \ldots, n+1\}.$ If $p+n$ hyperplanes in general position in $\mathbb{P}^{n}(C)$ have forward invariant preimages under $f$ with respect to the translation
$\tau(z)=z+c,$ then the image of $f$ is contained in a projective linear subspace over $\mathcal{P}_{c}^{1}$ of dimension $\leq [\frac{n}{p}].$
\end{theorem}

In this section we extend Theorem \ref{T-6.1} to the case of meromorphic mappings $f:\mathbb{C}^{m}\rightarrow\mathbb{P}^{n}(\mathbb{C})$ of hyperorder strictly less than one and hyperplanes in $N$-subgeneral position.\par

\begin{theorem}\label{T-6.2} Let $c\in \mathbb{C}^{m},$ let $p\in\{1, \ldots, \frac{N}{N-n+1}+N-n+1\},$ $n\leq N< n+p.$ Assume that $f$ is a meromorphic mapping from $\mathbb{C}^{m}$ into $\mathbb{P}^{n}(\mathbb{C})$ such that hyperorder $\zeta_{2}(f)<1.$ If $p+n$ hyperplanes in $N$-subgeneral position in $\mathbb{P}^{n}(C)$ have forward invariant preimages under $f$ with respect to the translation
$\tau(z)=z+c,$ then the image of $f$ is contained in a projective linear subspace over $\mathcal{P}_{c}^{1}$ of dimension $\leq [\frac{N}{n+p-N}-N+n].$
\end{theorem}

Before proving Theorem \ref{T-6.2}, we need two lemmas as follows. The first one is an extension of \cite[Lemma 3.3]{halburdkt:14}.\par

\begin{lemma}\label{L-4.1}
Let $c\in \mathbb{C}^{m},$ and $f=[f_{0}, \ldots, f_{n}]$ be a meromorphic mapping from $\mathbb{C}^{m}$ into $\mathbb{P}^{n}(\mathbb{C})$ such that hyperorder $\zeta_{2}(f)<\lambda\leq 1$ and all zeros of $f_{0}, \ldots, f_{n}$ are forward invariant with respect to the translation $\tau(z)=z+c.$ If $\frac{f_{i}}{f_{j}}\not\in \mathcal{P}^{\lambda}_{c}$ for all $i, j\in \{0,\ldots,n\}$ such that $i\neq j,$ then $f$ is linearly nondegenerate over $\mathcal{P}^{\lambda}_{c}.$
\end{lemma}

\begin{proof}
Assume that the conclusion is not true, that is there exist $A_{0}, \ldots, A_{n}\in \mathcal{P}^{\lambda}_{c}$ such that
$$A_{0}f_{0}+\cdots+A_{n-1}f_{n-1}=A_{n}f_{n}$$
and such that not all $A_{j}$ are identically zero. Without loss of generality we may assume that none of $A_{j}$ are identically zero. Since all zeros of $f_{0}, \ldots, f_{n}$ are forward invariant with respect to the translation $\tau(z)=z+c$ and since $A_{0}, \ldots, A_{n}\in \mathcal{P}^{\lambda}_{c},$ we can choose a meromorphic function $F$ on $\mathbb{C}^{m}$ such that $FA_{0}f_{0}, \ldots, FA_{n}f_{n}$ are holomorphic functions on $\mathbb{C}^{m}$ without common zeros and such that the preimages of all  zeros of $FA_{0}f_{0}, \ldots, FA_{n}f_{n}$ are forward invariant with respect to the translation $\tau(z)=z+c.$ Then we have \begin{eqnarray}\label{E-4.1}\limsup_{r\rightarrow\infty}\frac{\log^{+}\log^{+}\left(N(r, \nu_{F}^{0})+N(r, \nu_{F}^{\infty})\right)}{\log r}<\lambda\leq 1\end{eqnarray}
and $FA_{0}f_{0},\ldots, FA_{n-1}f_{n-1}$ cannot have any common zeros.\par

Denote $g_{j}:=FA_{j}f_{j}$ for $0\leq j\leq n.$ Then $T_{G}(r)$ is well defined for $G=[g_{0}, \ldots, g_{n-1}]$ which is a holomorphic mapping from $\mathbb{C}^{m}$ into $\mathbb{P}^{n-1}(\mathbb{C}).$ Then by the definition of characteristic function and the Jensen's theorem we have
\begin{eqnarray*}T_{G(r)}&=&\int_{S_{m}(r)}\log\|G\|\eta_{m}(z)+O(1)\\
&\leq&\int_{S_{m}(r)}\log |F|\eta_{m}(z)+\int_{S_{m}(r)}\log(\sum_{j=0}^{n-1}|f_{j}|^{2})^{\frac{1}{2}}\eta_{m}(z)\\
&&+\sum_{j=0}^{n-1}\int_{S_{m}(r)}\log^{+}|A_{j}|\eta_{m}(z)+O(1)\\
&\leq& N(r, \nu_{F}^{0})-N(r, \nu_{F}^{\infty})+T_{f}(r)+\sum_{j=0}^{n-1}T_{A_{j}}(r)
\end{eqnarray*}
which together with \eqref{E-4.1} imply that the hyperorder satisfies $\zeta_{2}(G)<\lambda\leq 1.$\par

Assume that the meromorphic mapping $G: \mathbb{C}^{m}\rightarrow\mathbb{P}^{n-1}(\mathbb{C})$ is linearly nondegenerate over $\mathcal{P}_{c}^{\lambda}(\subset \mathcal{P}_{c}).$ Then by Lemma \ref{L-3.4}, it follows that  $C(g_{0}, \ldots, g_{n-1})\not\equiv 0.$ Define the following hyperplanes
\begin{eqnarray*}
H_{0}:\quad &&  w_{0}=0,\\
H_{1}:\quad && w_{1}=0,\\
&\vdots&\\
H_{j}:\quad && w_{j}=0,\\
&\vdots&\\
H_{n-1}:\quad && w_{n-1}=0,\\
H_{n}:\quad && w_{0}+w_{1}+\ldots+w_{n-1}=0,
\end{eqnarray*}where $[w_{0}, \ldots, w_{n-1}]$ is a homogeneous coordinate system of $\mathbb{P}^{n-1}(\mathbb{C}).$ So, $(G, H_{j})=g_{j}$ for $0\leq j\leq n-1$ and $(G, H_{n})=g_{0}+\ldots+g_{n-1}=FA_{n}f_{n}=g_{n}.$ Obviously, the $q=n+1$ hyperplanes $H_{0}, \ldots, H_{n}$ are in $(n-1)$-subgeneral position of $\mathbb{P}^{n-1}(\mathbb{C}).$ Hence by Theorem \ref{T-3.1} we have
\begin{eqnarray*}T_{G}(r)&=&
\left((n+1)-2(n-1)+(n-1)-1\right)T_{G}(r)\\
&\leq& \sum_{j=0}^{n}N(r, \nu_{g_{j}}^{0})-N(r, \nu_{C(g_{0}, \ldots, g_{n-1})}^{0})+o(T_{G}(r))
\end{eqnarray*} for all $r$ outside of a possible exceptional set of finite logarithmic measure. Then using the same discussion as in the proof of \cite[Lemma 3.3]{halburdkt:14} we have
\begin{eqnarray*}
\sum_{j=1}^{n}N(r, \nu_{g_{j}}^{0})\leq N(r, \nu_{C(g_{0}, \ldots, g_{n-1})}^{0}).
\end{eqnarray*}
Hence, it follows $T_{G}(r)=o(T_{G}(r))$ which is an contradiction.\par

Therefore, the meromorphic mapping $G: \mathbb{C}^{m}\rightarrow\mathbb{P}^{n-1}(\mathbb{C})$ is linearly degenerate over $\mathcal{P}_{c}^{\lambda},$ and thus there exist $B_{0}, \ldots, B_{n-1}\in \mathcal{P}^{\lambda}_{c}$ such that
$$B_{0}f_{0}+\cdots+B_{n-2}f_{n-2}=B_{n-1}f_{n-1}$$
and such that not all $B_{j}$ are identically zero. By repeating similar discussions as above it follows that there exist $L_{i}, L_{j}\in\mathcal{P}_{c}^{\lambda}$ such that $$L_{i}f_{i}=L_{j}f_{j}$$ for some $i\neq j$ and not all $L_{i}$ and $L_{j}$ are identically zero. This contradicts
the condition that $\frac{f_{i}}{f_{j}}\not\in \mathcal{P}^{\lambda}_{c}$ for all $\{i,j\}\subset\{0,\ldots,n\}.$  Therefore, the proof is complete.
\end{proof}

The following lemma is an extension of the difference analogue of Borel's theorem \cite[Theorem 3.1]{halburdkt:14}.\par

\begin{lemma}\label{L-4.2}
Let $c\in\mathbb{C}^{m},$ and $f=[f_{0}, \ldots, f_{n}]$ be a meromorphic mapping from $\mathbb{C}^{m}$ into $\mathbb{P}^{n}(\mathbb{C})$ such that hyperorder $\zeta_{2}(f)<\lambda\leq 1$  and all zeros of $f_{0}, \ldots, f_{n}$ are forward invariant with respect to the translation $\tau(z)=z+c.$ Let
$$S_{1}\cup\cdots\cup S_{l}$$ be the partition of $\{0, 1, \ldots, n\}$ formed in such a way that $i$ and $j$ are
in the same class $S_{k}$ if and only if $\frac{f_{i}}{f_{j}}\in \mathcal{P}^{\lambda}_{c}.$ If
$$f_{0}+\ldots+f_{n}=0,$$
then $$\sum_{j\in S_{k}}f_{j}=0$$ for all $k\in\{1,\ldots,l\}.$
\end{lemma}

\begin{proof} Suppose that $i\in S_{k},$ $k\in\{0,\ldots, l\}.$ Then by the condition of the lemma, $f_{i}=A_{i, j_{k}}f_{j_{k}}$ for some $A_{i, j_{k}}\in\mathcal{P}_{c}^{\lambda}$ whenever the indexes $i$ and $j_{k}$ are in the same class $S_{k}.$ This implies that
$$0=\sum_{k=0}^{n}f_{k}=\sum_{k=1}^{l}\sum_{i\in S_{k}}A_{i,j_{k}}f_{j_{k}}=\sum_{k=1}^{l}B_{k}f_{j_{k}}$$ where $B_{k}=\sum_{i\in S_{k}}A_{i,j_{k}}\in\mathcal{P}_{c}^{\lambda}.$ This says that $f_{j_{1}}, \ldots, f_{j_{l}}$ are linearly degenerate over $\mathcal{P}_{c}^{\lambda}.$ Hence by Lemma \ref{L-4.1} all $B_{k}$ $(k=1,\ldots, l)$ are identically zero. Thus it follows
$$\sum_{i\in S_{k}}f_{i}=\sum_{i\in S_{k}}A_{i,j_{k}}f_{j_{k}}=B_{k}f_{j_{k}}\equiv 0$$
for all $k=\{1,\ldots ,l\}.$
\end{proof}

\begin{proof}[Proof of Theorem \ref{T-6.2}]
Denote $Q=\{1, \ldots, n+p\}.$ Let $H_{j}$ be defined by
$$H_{j}:\quad h_{j0}(z)w_{0}+\ldots+h_{jn}(z)w_{n}=0\quad (j\in Q)$$ where $[w_{0}, \ldots, w_{n}]$ is a homogeneous coordinate system of $\mathbb{P}^{n}(\mathbb{C}).$ Since $\{H_{j}\}_{j\in Q}$ are in $N$-subgeneral position of $\mathbb{P}^{n}(\mathbb{C}),$ any $N+2$ of $H_{j}$ satisfy a linear relation with nonzero coefficients in $\mathbb{C}^{1}.$ By conditions of the theorem, holomorphic functions
$$g_{j}:=(f, H_{j})=h_{j0}f_{0}+\ldots+h_{jn}f_{n}$$ satisfy
$$\{\tau(g_{j}^{-1}(\{0\}))\}\subset \{g_{j}^{-1}(\{0\})\}$$ for all $j\in Q,$ where $\{\cdot\}$ denotes a multiset with counting multiplicities of its  elements. We say that $i\sim j$ if $g_{i}=\alpha g_{j}$ for some $\alpha\in\mathcal{P}^{1}_{c}\setminus\{0\}.$  Hence
$$Q=\bigcup_{j=1}^{l}S_{j}$$ for some $l\in Q.$\par

Firstly, assume that the complement of $S_{k}$ has at least $N+1$ elements for some $k\in\{1, \ldots l\}.$ Choose an element $s_{0}\in S_{k},$ and denote $U=(Q\setminus S_{k})\cup\{s_{0}\}.$ Then $U$ contains at least $N+2$ elements, and thus there is a subset $U_{0}\subset U$ such that $U_{0}\cap S_{k}=\{s_{0}\}$ and $|U_{0}|=N+2.$ Therefore there exists $\alpha_{j}\in\mathbb{C}\setminus\{0\}$ such that
$$\sum_{j\in U_{0}}\alpha_{j}H_{j}=0.$$ Hence,
 $$\sum_{j\in U_{0}}\alpha_{j}g_{j}=\sum_{j\in U_{0}}\alpha_{j}(f, H_{j})=\sum_{j\in U_{0}}\alpha_{j}H_{j}(f) =0.$$

Without loss of generality, we may assume that $U_{0}=\{s_{1}, \ldots, s_{N+1}\}\cup\{s_{0}\}.$ It is easy to see from above discussion that all of zeros of $\alpha_{j}g_{j}$ $(j\in U_{0})$ are forward invariant with respect to the translation $\tau(z)=z+c,$ and $$G:=[\alpha_{s_{0}}g_{s_{0}}: \alpha_{s_{1}}g_{s_{1}}: \cdots: \alpha_{s_{N+1}}g_{s_{N+1}}]$$ is a meromorphic mapping from $\mathbb{C}^{m}$ into $\mathbb{P}^{N+1}(\mathbb{C})$ with its hyperorder $\zeta_{2}(G)\leq \zeta_{2}(f)<1.$ Furthermore, $\frac{\alpha_{i}g_{i}}{\alpha_{s_{0}}g_{s_{0}}}\not\in\mathcal{P}_{c}^{1}$ for any $i\in U_{0}\setminus\{s_{0}\},$ thus $i\not\sim s_{0}.$ Hence by Lemma \ref{L-4.2}  we have $\alpha_{s_{0}}g_{s_{0}}=0,$ and thus $(f, H_{s_{0}})\equiv 0.$ This means that the image $f(\mathbb{C}^{m})$ is included in the hyperplane $H_{s_{0}}$ of $\mathbb{P}^{n}(\mathbb{C}).$\par

Secondly, assume that the set $Q\setminus S_{k}$ has at most $N$ elements. Then $S_{k}$ has at least $n+p-N$ elements for all $k=1, \ldots, l.$ This implies that $$l\leq \frac{n+p}{n+p-N}.$$\par

Let $V$ be any subset of $Q$ with $|V|=N+1.$ Then $\{H_{j}\}_{j\in V}$ are linearly independent. Denote $V_{k}:=V\cap S_{k}.$ Then we have
$$V=\bigcup_{k=1}^{l}V_{k}.$$  Since each set $V_{k}$ gives raise to $|V_{k}-1|$ equations over the field $\mathcal{P}_{c}^{1},$ it follows that there are  at least
\begin{eqnarray*}\sum_{k=1}^{l}(|V_{k}|-1)&=&N+1-l\geq N+1-\frac{n+p}{n+p-N}\\&=&n-(n-N+\frac{N}{n+p-N})\end{eqnarray*}
 linear independent relations over the field $\mathcal{P}_{c}^{1}.$ This means that the image of $f$ is contained in a linear subspace over $\mathcal{P}_{c}^{1}$ of dimension $\leq [\frac{N}{n+p-N}-N+n].$ The proof of the theorem is complete.\end{proof}

\begin{proof}[Proof of Theorem \ref{T-5.3}] By Theorem \ref{T-6.2},  the image of $f$ is contained in a projective linear subspace over $\mathcal{P}_{c}^{1}$ of dimension $\leq [\frac{N}{n+p-N}-N+n].$
By the assumption $p>\frac{N}{N-n+1}+N-n$ it follows  $[\frac{N}{n+p-N}-N+n]=0.$ Hence $f(z)=f(z+c).$ The proof of Theorem \ref{T-5.3} is thus complete.\end{proof}


\def\cprime{$'$}


\begin{thebibliography}{10}
\expandafter\ifx\csname url\endcsname\relax
  \def\url#1{\texttt{#1}}\fi
\expandafter\ifx\csname urlprefix\endcsname\relax\def\urlprefix{URL }\fi
\expandafter\ifx\csname href\endcsname\relax
  \def\href#1#2{#2} \def\path#1{#1}\fi

\bibitem{nevanlinna:25}
R.~Nevanlinna, Zur {T}heorie der
  {M}eromorphen {F}unktionen, Acta Math. 46~(1-2) (1925) 1--99.
\newblock \href {http://dx.doi.org/10.1007/BF02543858}
  {\path{doi:10.1007/BF02543858}}.
\newline\urlprefix\url{http://dx.doi.org/10.1007/BF02543858}

\bibitem{chen:90}
W.~X. Chen, Defect relations for
  degenerate meromorphic maps, Trans. Amer. Math. Soc. 319~(2) (1990)
  499--515.
\newblock \href {http://dx.doi.org/10.2307/2001251}
  {\path{doi:10.2307/2001251}}.
\newline\urlprefix\url{http://dx.doi.org/10.2307/2001251}

\bibitem{cartan:33}
H.~Cartan, Sur l{\'e}s zeros des combinaisons lin{\'e}aires de $p$ fonctions
  holomorphes donn{\'e}es, Mathematica Cluj 7 (1933) 5--31.

\bibitem{halburdk:06AASFM}
R.~G. Halburd, R.~J. Korhonen, Nevanlinna theory for the difference operator,
  Ann. Acad. Sci. Fenn. Math. 31~(2) (2006) 463--478.

\bibitem{wonglw:09}
P.-M. Wong, H.-F. Law, P.~P.~W. Wong,
  A second main theorem on
  {$\Bbb P^n$} for difference operator, Sci. China Ser. A 52~(12) (2009)
  2751--2758.
\newblock \href {http://dx.doi.org/10.1007/s11425-009-0213-5}
  {\path{doi:10.1007/s11425-009-0213-5}}.
\newline\urlprefix\url{http://dx.doi.org/10.1007/s11425-009-0213-5}

\bibitem{halburdkt:14}
R.~Halburd, R.~Korhonen, K.~Tohge,
  Holomorphic curves
  with shift-invariant hyperplane preimages, Trans. Amer. Math. Soc. 366~(8)
  (2014) 4267--4298.
\newblock \href {http://dx.doi.org/10.1090/S0002-9947-2014-05949-7}
  {\path{doi:10.1090/S0002-9947-2014-05949-7}}.
\newline\urlprefix\url{http://dx.doi.org/10.1090/S0002-9947-2014-05949-7}

\bibitem{korhonen:12}
R.~Korhonen, A difference {P}icard
  theorem for meromorphic functions of several variables, Comput. Methods
  Funct. Theory 12~(1) (2012) 343--361.
\newblock \href {http://dx.doi.org/10.1007/BF03321831}
  {\path{doi:10.1007/BF03321831}}.
\newline\urlprefix\url{http://dx.doi.org/10.1007/BF03321831}

\bibitem{cao:14}
T.-B. Cao, Difference analogues
  of the second main theorem for meromorphic functions in several complex
  variables, Math. Nachr. 287~(5-6) (2014) 530--545.
\newblock \href {http://dx.doi.org/10.1002/mana.201200234}
  {\path{doi:10.1002/mana.201200234}}.
\newline\urlprefix\url{http://dx.doi.org/10.1002/mana.201200234}

\bibitem{noguchi:05}
J.~Noguchi, A note on entire
  pseudo-holomorphic curves and the proof of {C}artan-{N}ochka's theorem,
  Kodai Math. J. 28~(2) (2005) 336--346.
\newblock \href {http://dx.doi.org/10.2996/kmj/1123767014}
  {\path{doi:10.2996/kmj/1123767014}}.
\newline\urlprefix\url{http://dx.doi.org/10.2996/kmj/1123767014}

\bibitem{fujimoto:72b}
H.~Fujimoto, On holomorphic maps into a taut complex space, Nagoya Math. J. 46
  (1972) 49--61.

\bibitem{green:72}
M.~L. Green, Holomorphic maps into complex projective space omitting
  hyperplanes, Trans. Amer. Math. Soc. 169 (1972) 89--103.

\bibitem{fujimoto:74}
H.~Fujimoto, On meromorphic maps into the complex projecive space, J. Math.
  Soc. Japan 26 (1974) 272--288.

\bibitem{fujimoto:93}
H.~Fujimoto, Value
  distribution theory of the {G}auss map of minimal surfaces in {${\bf R}^m$},
  Aspects of Mathematics, E21, Friedr. Vieweg \& Sohn, Braunschweig, 1993.
\newblock \href {http://dx.doi.org/10.1007/978-3-322-80271-2}
  {\path{doi:10.1007/978-3-322-80271-2}}.
\newline\urlprefix\url{http://dx.doi.org/10.1007/978-3-322-80271-2}

\bibitem{nochka:83}
E.~Nochka, {On the theory of meromorphic functions.}, Sov. Math., Dokl. 27
  (1983) 377--381.

\bibitem{fujimoto:85}
H.~Fujimoto, Nonintegrated defect relation for meromorphic maps of complete
  {K}\"ahler manifolds into {$P^{N_1}({\bf C})\times\cdots\times P^{N_k}({\bf
  C})$}, Japan. J. Math. (N.S.) 11~(2) (1985) 233--264.

\bibitem{goldbergo:08}
A.~A. Goldberg, I.~V. Ostrovskii, Value distribution of meromorphic functions,
  Vol. 236 of Translations of Mathematical Monographs, American Mathematical
  Society, Providence, RI, 2008, translated from the 1970 Russian original by
  Mikhail Ostrovskii, With an appendix by Alexandre Eremenko and James K.
  Langley.

\bibitem{stoll:64}
W.~Stoll, Normal families of non-negative divisors, Math. Z. 84 (1964)
  154--218.

\bibitem{fujimoto:74Nagoya}
H.~Fujimoto, On families of meromorphic maps into the complex projective space,
  Nagoya Math. J. 54 (1974) 21--51.

\bibitem{hinkkanen:92}
A.~Hinkkanen, A sharp form of
  {N}evanlinna's second fundamental theorem, Invent. Math. 108~(3) (1992)
  549--574.
\newblock \href {http://dx.doi.org/10.1007/BF02100617}
  {\path{doi:10.1007/BF02100617}}.
\newline\urlprefix\url{http://dx.doi.org/10.1007/BF02100617}

\bibitem{fujimoto:98}
H.~Fujimoto, Uniqueness
  problem with truncated multiplicities in value distribution theory, Nagoya
  Math. J. 152 (1998) 131--152.
\newline\urlprefix\url{http://projecteuclid.org/euclid.nmj/1118766414}

\bibitem{ru:01}
M.~Ru, Nevanlinna theory and its
  relation to {D}iophantine approximation, World Scientific Publishing Co.,
  Inc., River Edge, NJ, 2001.
\newblock \href {http://dx.doi.org/10.1142/9789812810519}
  {\path{doi:10.1142/9789812810519}}.
\newline\urlprefix\url{http://dx.doi.org/10.1142/9789812810519}

\bibitem{nevanlinna:26}
R.~Nevanlinna, Einige
  {E}indeutigkeitss\"atze in der {T}heorie der {M}eromorphen {F}unktionen,
  Acta Math. 48~(3-4) (1926) 367--391.
\newblock \href {http://dx.doi.org/10.1007/BF02565342}
  {\path{doi:10.1007/BF02565342}}.
\newline\urlprefix\url{http://dx.doi.org/10.1007/BF02565342}

\bibitem{fujimoto:75}
H.~Fujimoto, The uniqueness problem of meromorphic maps into the complex
  projective space, Nagoya Math. J. 58 (1975) 1--23.

\bibitem{yangy:03}
C.-C. Yang, H.-X. Yi,
  Uniqueness theory of
  meromorphic functions, Vol. 557 of Mathematics and its Applications, Kluwer
  Academic Publishers Group, Dordrecht, 2003.
\newblock \href {http://dx.doi.org/10.1007/978-94-017-3626-8}
  {\path{doi:10.1007/978-94-017-3626-8}}.
\newline\urlprefix\url{http://dx.doi.org/10.1007/978-94-017-3626-8}

\bibitem{huly:03}
P.-C. Hu, P.~Li, C.-C. Yang,
  Unicity of meromorphic
  mappings, Vol.~1 of Advances in Complex Analysis and its Applications,
  Kluwer Academic Publishers, Dordrecht, 2003.
\newblock \href {http://dx.doi.org/10.1007/978-1-4757-3775-2}
  {\path{doi:10.1007/978-1-4757-3775-2}}.
\newline\urlprefix\url{http://dx.doi.org/10.1007/978-1-4757-3775-2}

\end{thebibliography}

\end{document}